\documentclass[twoside,12pt,leqno]{amsart}
\usepackage{pstricks,amssymb,latexsym,verbatim} % verbatim for ``comment''
\usepackage{hyperref}
\hypersetup{citecolor=[rgb]{1,0,0,}, linkcolor=[rgb]{0,0,1}, colorlinks=true}
\theoremstyle{plain}
\newtheorem{theorem}{Theorem}%[section]
\newtheorem{lemma}[theorem]{Lemma}
\newtheorem{proposition}[theorem]{Proposition}
\newtheorem{corollary}[theorem]{Corollary}

\newtheorem{notation}[theorem]{Notation}

\newcommand{\Char}{\textup{char}}
\renewcommand{\geq}{\geqslant}
\renewcommand{\leq}{\leqslant}

\newcommand{\F}{\mathbb{F}}

\begin{document}

\newcommand{\lcm}{\mathrm{lcm}}

\hyphenation{}

\title[`Jordan partitions' and their parts and \texorpdfstring{{\it
p}}{}-parts]{Decomposing modular tensor products:\\`Jordan
partitions', their parts and \texorpdfstring{{\it p}}{}-parts}
\author{S.\,P. Glasby, Cheryl E. Praeger, and Binzhou Xia}

\address[Glasby]{ Centre for Mathematics of Symmetry and Computation\\
University of Western Australia\\ 35 Stirling Highway\\ Crawley 6009,
Australia. Also affiliated with The Faculty of Mathematics and Technology,
University of Canberra, ACT 2601, Australia. Email:
{\tt GlasbyS@gmail.com; WWW:
\href{http://www.maths.uwa.edu.au/~glasby/}{http://www.maths.uwa.edu.au/$\sim$glasby/}}}
\address[Praeger] {Centre for Mathematics of Symmetry and
Computation\\ University of Western Australia\\ 35 Stirling Highway\\
Crawley 6009, Australia. Also affiliated with King Abdulaziz
University, Jeddah, Saudi Arabia. Email: {\tt
Cheryl.Praeger@uwa.edu.au;\newline WWW:
\href{http://www.maths.uwa.edu.au/~praeger}{http://www.maths.uwa.edu.au/$\sim$praeger}
}} \address[Xia]{ Centre for Mathematics of Symmetry and Computation\\
University of Western Australia\\ 35 Stirling Highway\\ Crawley 6009,
Australia. Current address: School of Mathematical Science, Peking
University, Beijing, People's Republic of China.\newline Email: {\tt
BinzhouXia@pku.edu.cn} }

\date{\today}

\begin{abstract} Determining the Jordan canonical form of the tensor
product of Jordan blocks has many applications including to the
representation theory of algebraic groups, and to tilting modules.
Although there are several algorithms for computing this decomposition
in literature, it is difficult to predict the output of these
algorithms. We call a decomposition of the form $J_r\otimes
J_s=J_{\lambda_1}\oplus\cdots\oplus J_{\lambda_b}$ a `Jordan
partition'. We prove several deep results concerning the $p$-parts of
the $\lambda_i$ where $p$ is the characteristic of the underlying
field. Our main results include the proof of two conjectures made by
McFall in 1980, and the proof that $\lcm(r,s)$ and
$\gcd(\lambda_1,\dots,\lambda_b)$ have equal $p$-parts. Finally, we
establish some explicit formulas for Jordan partitions when $p=2$.
\end{abstract}

\maketitle \centerline{\noindent AMS Subject Classification (2010):
15A69, 15A21, 13C05}

\section{Introduction}

Throughout this paper $F$ denotes a field with
characteristic~$p\geq0$. Given $\alpha\in F$ denote by $J_r(\alpha)$
the $r\times r$ Jordan block with eigenvalue $\alpha$. Hence $(\alpha
I-J_r(\alpha))^k=0$ holds if and only if $k\geq r$. Given
$\alpha,\beta\in F$ and $r,s\geq1$ the Jordan canonical form of the
tensor product $J_r(\alpha)\otimes J_s(\beta)$ equals
$J_{\lambda_1}(\alpha\beta)\oplus\cdots\oplus
J_{\lambda_b}(\alpha\beta)$ where $rs=\lambda_1+\cdots+\lambda_b$. The
partition $(\lambda_1,\dots,\lambda_b)$ of $rs$ is easily described
when $\alpha\beta=0$, see for example
\cite[Prop.\;2.1.2]{II2009}. When $\alpha\beta\neq0$, a simple change
of basis shows that the corresponding partition is the same as that
for $J_r(1)\otimes J_s(1)$. We denote it by $\lambda(r,s,p)$ as the
Jordan canonical form of $J_r(1)\otimes J_s(1)$ is invariant under
field extensions. We call $\lambda(r,s,p)=(\lambda_1,\dots,\lambda_b)$
a `Jordan partition' and always write its parts in non-increasing
order $\lambda_1\geq\dots\geq\lambda_b>0$. It has been long known that
$b$ equals $\min(r,s)$, see \cite[Lemma\;2.1]{Ralley1969}. Note that
$\lambda(r,s,p)=\lambda(s,r,p)$ since $J_r(1)\otimes J_s(1)$ is
similar to $J_s(1)\otimes J_r(1)$.

The partition $\lambda(r,s,p)$ is well known if $\Char(F)=0$, or
$\Char(F)=p\geq r+s-1$. In these cases, the $i$th part of
$\lambda(r,s,p)$ is $\lambda_i=r+s+1-2i$,
see~\cite[Corollary~1]{Srinivasan1964}. Henceforth, we will assume
that $\Char(F)=p$ is an arbitrary prime, possibly satisfying $p\geq
r+s-1$.  The {\it $p$-part} of a nonzero integer $n$, denoted by
$n_p$, is the largest $p$-power dividing $n$.

There is a well-known link between the partition $\lambda(r,s,p)$ and
the modular representation theory of a cyclic group $C_{p^n}$ of order
$p^n$ where $\max(r,s)\leq p^n$. There are precisely $p^n$ pairwise
nonisomorphic indecomposable $FC_{p^n}$-modules, say
$V_1,\dots,V_{p^n}$ where $\dim(V_i)=i$. In his pioneering
work~\cite{Green1962}, Green studied a ring, now called the
\emph{modular representation ring} or \emph{Green ring}, whose
elements are $F$-linear combinations $\sum_{i=1}^{p^n}\alpha_i[V_i]$
of the isomorphism classes $[V_i]$.  Addition and multiplication are
given by the direct sum and by tensor product, and denoted $\oplus$
and $\otimes$. It is conventional to write the module $V_i$ instead of
the isomorphism class $[V_i]$, and to let $V_0$ be a 0-dimensional
module. As usual $mV$ denotes the direct sum of $m$ copies of $V$
where $m\geq0$ is an integer. Thus $0V$ is just the zero module, and
$[V_0]=[0V]$. Given positive integers $r,s$ satisfying $r,s\leq p^n$,
the module $V_r\otimes V_s$ is a sum of indecomposable modules by the
Krull-Schmidt theorem. This gives a Green ring equation
\begin{equation}\label{eq7} V_r\otimes
V_s=V_{\lambda_1}\oplus\cdots\oplus V_{\lambda_b}
\quad\textup{where}\quad b:=\min(r,s),
\end{equation} where the parts $\lambda_i$ of the partition
$\lambda(r,s,p)=(\lambda_1,\dots,\lambda_b)$ are at most~$p^n$. It is
easy to convert between the Green ring decomposition (\ref{eq7}) and
the partition $\lambda(r,s,p)$, and we shall do so frequently in this
paper.

Given positive integers $r$ and $s$, let $p^n$ be the smallest
$p$-power exceeding $\max(r,s)$. A fundamental question is how to
decompose $V_r\otimes V_s$ as $V_{\lambda_1}\oplus\dots\oplus
V_{\lambda_b}$. In fact, a majority of papers addressing Jordan
partitions in the literature were concerned with this decomposition
problem, and there are basically two classes of algorithms. One class
of algorithms~\cite{McFall1979,Norman1995,Renaud1979} involves
recursive computations to reduce~$n$. Although these algorithms are
similar in spirit, the one proposed by Renaud~\cite{Renaud1979} in
1979 is more convenient to apply, and we use it repeatedly in
Section~\ref{sec1}. The other class of
algorithms~\cite{II2009,McFall1980,Ralley1969,Srinivasan1964} is
related to binomial matrices (matrices of binomial coefficients). Iima
and Iwamatsu~\cite{II2009} presented a novel algorithm which, unlike
it predecessors, avoided the computation of ranks of binomial matrices
over $\mathbb{F}_p$, called $p$-ranks. In 2009, Iima and
Iwamatsu~\cite{II2009} showed that, to compute the parts of
$\lambda(r,s,p)$, it suffices to know whether or not the $p$-ranks of
certain binomial matrices are full. This reduces the computation
dramatically since the determinants of those binomial matrices can be
computed via an explicit formula, and we can study their
$p$-divisibility using number theory. For complementary introductory
remarks, see~\cite[\S1]{GPX}.

There are, however, relatively few results on the properties of the decomposition, or the partition in the literature. The following one is due to Green~\cite{Green1962}, who assumed the $\lambda_i$ to be positive. It is convenient for us to assume that each part is nonnegative.

\begin{proposition}\label{prop1}\textup{\cite[(2.5a)]{Green1962}}
Suppose $1\leq r,s\leq p^n$. If $V_r\otimes V_s=V_{\lambda_1}\oplus\dots\oplus V_{\lambda_b}$, then
$$
V_{p^n-r}\otimes V_s=(s-b)V_{p^n}\oplus V_{p^n-\lambda_b}\oplus\dots\oplus V_{p^n-\lambda_1}\quad\textup{where}\footnotemark\footnotetext{Surprisingly, the fact that $b=\min(r,s)$ does not appear in~\cite{Green1962}. Its proof is easy, see~\cite[Lemma\;9(a)]{GPX}.}\quad b=\min(r,s).
$$
\end{proposition}

Proposition~\ref{prop1} can be a viewed as a `duality' result on $\lambda(r,s,p)$. For more on this duality and some `periodicity' results as well as other properties, the reader is referred to~\cite{GPX}. In this paper, the main results are Theorem~\ref{thm1}, which was described in the abstract, and Theorems~\ref{thm2} and~\ref{thm5}, which were conjectured by McFall~\cite[p.\,87]{McFall1980} using different notation. We also prove in Section~\ref{sec4} some results about the $p$-parts of $\lambda_1,\dots,\lambda_b$ when $|r-s|\leq1$, and prove explicit decomposition formulas when $p=2$. Some of these later results were foreshadowed by McFall~\cite[Theorem~2]{McFall1979} who gave an algorithm for computing the Jordan decomposition when $p=2$.

\begin{theorem}\label{thm1}
Suppose $V_r\otimes V_s=\bigoplus_{i=1}^b V_{\lambda_i}$ where $\Char(F)=p$ is prime.
Then the $p$-parts of $\lcm(r,s)$ and $\gcd(\lambda_1,\dots,\lambda_b)$ are
equal. That is, $\lcm(r,s)_p=\gcd(\lambda_1,\dots,\lambda_b)_p.$
\end{theorem}

\begin{notation}[Multiplicity]\label{notn} Write
$V_r\otimes V_s=\bigoplus_{i=1}^b V_{\lambda_i}$ as
$V_r\otimes V_s=\bigoplus_{i=1}^t m_iV_{\mu_i}$ where the multiset
$\{\lambda_1,\dots,\lambda_b\}$ has distinct parts $\mu_1>\cdots>\mu_t>0$,
which occur with positive multiplicities $m_1,\dots,m_t$, respectively.
\end{notation}

If~$t$ is much smaller than~$b$, it can be helpful to write
$\bigoplus_{i=1}^t m_iV_{\mu_i}$ instead of~$\bigoplus_{i=1}^b V_{\lambda_i}$.
Observe that $\sum_{i=1}^t m_i=b$ and
$\sum_{i=1}^t m_i\mu_i=\sum_{i=1}^b \lambda_i=rs$. We will commonly switch
between the parts
$\mu_i$ and the corresponding summand $V_{\mu_i}$ with $\dim(V_{\mu_i})=\mu_i$.
Since $\gcd(\lambda_1,\dots,\lambda_b)$ equals $\gcd(\mu_1,\dots,\mu_t)$
and $\lcm(r,s)_p$ equals $\max(r_p,s_p)$, we see that
\begin{equation}\label{X}
  \gcd(\lambda_1,\dots,\lambda_b)_p=\gcd(\mu_1,\dots,\mu_t)_p
  =\gcd((\mu_1)_p,\dots,(\mu_t)_p)=\min((\mu_1)_p,\dots,(\mu_t)_p).
\end{equation}
Using multiplicities as described in Notation~\ref{notn}, we paraphrase
Theorem~\ref{thm1} as follows:
\begin{equation}\label{M}
  V_r\otimes V_s=\bigoplus_{i=1}^t m_iV_{\mu_i}\quad\textup{implies}\quad
  \max(r_p,s_p)=\min((\mu_1)_p,\dots,(\mu_t)_p).
\end{equation}
In 1980, McFall made two conjectures, see p.\;87 of~\cite{McFall1980}. His first
conjecture is proved by Theorem~\ref{thm2} below. His second conjecture
is implied by the formula~\eqref{C2} in Theorem~\ref{thm5}.

\begin{theorem}\label{thm2}
Suppose that $r,s\geq1$ and $V_r\otimes V_s=\bigoplus_{i=1}^tm_iV_{\mu_i}$ where
the summands are nonzero and the $\mu_i$ are distinct. If a multiplicity
satisfies $m_i>1$, then~$\mu_i$ is divisible by~$p$.
\end{theorem}

\begin{theorem}\label{thm5}
Suppose that $r,s\geq1$ and $V_r\otimes V_s=\bigoplus_{i=1}^tm_iV_{\mu_i}$  as in Notation~\textup{\ref{notn}}. Then the
multiplicities $m_1,\dots,m_t$ determine the
part sizes $\mu_1>\cdots>\mu_t>0$, and conversely,~via
\begin{align}
  \mu_i&=r+s-m_i-2\sum_{j=1}^{i-1} m_j&&\textup{for $1\leq i\leq t$,}\label{E:McF} \\
  m_i&=(-1)^{i-1}\left[r+s+2\sum_{j=1}^{i-1}\mu_j\right]-\mu_i&&\textup{for $1\leq i\leq t$}.\label{C2}
\end{align}
\end{theorem}

These results have several simple consequences. We mention just one.
Theorem~\ref{thm2} says $p\nmid\mu_1$ implies $m_1=1$, and Theorem~\ref{thm5}
says $m_1=r+s-\mu_1$. Hence $p\nmid\mu_1$ implies $r+s\not\equiv1\pmod p$.
In many fields, theoretical development precedes and informs algorithmic
development. In this field the reverse seems to hold. While algorithms
such as those in~\cite{II2009,McFall1979,Norman1995,Renaud1979} are helpful for
computing Jordan partitions, predicting the output for given input of $r,s,p$
is not at all obvious. Our hope is that the patterns in
Theorems~\ref{thm1},\,\ref{thm2},\,\ref{thm5} that we prove by appealing to
various algorithms may lead, in turn, to simpler, or more efficient, algorithms
for computing Jordan partitions.

The layout of this paper is as follows. Renaud's decomposition algorithm is reviewed in Section~\ref{sec1}, and it is used to prove Theorems~\ref{thm1} and~\ref{thm2} in Section~\ref{sec2}. Section~\ref{sec3} introduces a different decomposition algorithm by Iima and Iwamatsu, and it is used to prove Theorem~\ref{thm5}. In the final section~\ref{sec4},  we establish some new results when $|r-s|\leq1$.

\section{Renaud's Algorithm}\label{sec1}

It is convenient to view $V_r$ as a module for all cyclic
groups $C_{p^n}$ with $p^n\geq r$.
Renaud's algorithm~\cite{Renaud1979} uses induction on $n$ to
decompose $V_r\otimes V_s$ where $n$ is the smallest integer
satisfying $\max(r,s)<p^n$ and $\Char(F)=p$. The inductive step is
achieved by the somewhat complicated reduction formula in
Proposition~\ref{prop2}. (The base case when $n=1$ is described in
Proposition~\ref{prop3}.) Note that the summand
$V_{(s_0-r_0)p^n+\nu_j}$ in~\cite[\textsc{Theorem~2}]{Renaud1979} is
incorporated as the $i=0$ summand on the third line of
equation~\eqref{eq4}.

\begin{proposition}\label{prop2}\textup{\cite[\textsc{Theorem~2}]{Renaud1979}} Suppose $1\leq r\leq s<p^{n+1}$ where $n\geq1$. Write $r=r_0p^n+r_1$ and $s=s_0p^n+s_1$, where $r_0,s_0,r_1,s_1\geq0$ and $r_1,s_1<p^n$. Suppose the decomposition $V_{r_1}\otimes V_{s_1}=\bigoplus_{j=1}^\ell n_jV_{\nu_j}$ has $p^n\geq\nu_1>\cdots>\nu_\ell>0$ and each $n_j>0$. Then
\begin{eqnarray}\label{eq4}
V_r\otimes V_s&=&cV_{p^{n+1}}\,\oplus\,|r_1-s_1|\bigoplus_{i=1}^{d_1}V_{(s_0-r_0+2i)p^n}\,\oplus\,\max(0,r_1-s_1)V_{(s_0-r_0)p^n}\nonumber\\
&&\oplus\;(p^n-r_1-s_1)\bigoplus_{i=1}^{d_2}V_{(s_0-r_0+2i-1)p^n}\\
&&\oplus\;\bigoplus_{j=1}^\ell n_j\left(\bigoplus_{i=0}^{d_1}V_{(s_0-r_0+2i)p^n+\nu_j}\,\oplus\,\bigoplus_{i=1}^{d_1}V_{(s_0-r_0+2i)p^n-\nu_j}\right)\nonumber,
\end{eqnarray}
where
$$
(c,d_1,d_2)=
\begin{cases}
(0,r_0,r_0)\quad&\text{if $r_0+s_0<p$},\\
(r+s-p^{n+1},p-s_0-1,p-s_0)\quad&\text{if $r_0+s_0\geq p$}.
\end{cases}
$$
\end{proposition}

Observe that~\eqref{eq4} fails to be a decomposition only when the multiplicity $p^n-r_1-s_1$ on the second line of~\eqref{eq4} is negative. However, in this case the whole second line cancels with some terms on the third line; see the remarks following Lemma~\ref{lem1}. To see how cancellation occurs in the Green ring to obtain a decomposition, we need a lemma.

\begin{lemma}\label{lem1}
%If $1\leq r_1,s_1\leq p^n$ and $r_1+s_1>p^n$, then the largest part of $\lambda(r_1,s_1,p)$ is $p^n$ and it occurs with multiplicity $r_1+s_1-p^n$. 
Suppose $r_1,s_1$ are positive integers satisfying $r_1,s_1\leq p^n$ and $r_1+s_1>p^n$. Then the largest part of $\lambda(r_1,s_1,p)$ is $p^n$, and it occurs with multiplicity $r_1+s_1-p^n$. 
%Rephrasing the conclusion using multiplicity notation $V_{r_1}\otimes V_{s_1}=\bigoplus_{j=1}^l n_jV_{\nu_j}$ (Notation~\textup{\ref{notn}}), says $\nu_1=p^n$ and $n_1=r_1+s_1-p^n$.
That is, if $V_{r_1}\otimes V_{s_1}=\bigoplus_{j=1}^l n_jV_{\nu_j}$  using \textup{Notation~\ref{notn}}, then $\nu_1=p^n$ and $n_1=r_1+s_1-p^n$.
\end{lemma}

\begin{proof}
By our assumption, $p^n-r_1<s_1$ and $\min(p^n-r_1,s_1)=p^n-r_1$. Suppose that $\lambda(p^n-r_1,s_1,p)=(\lambda_1,\dots,\lambda_{p^n-r_1})$; equivalently $V_{p^n-r_1}\otimes V_{s_1}=V_{\lambda_1}\oplus\dots\oplus V_{\lambda_{p^n-r_1}}$, where $\lambda_1\geq\dots\geq\lambda_{p^n-r_1}>0$. Then by Proposition~\ref{prop1},
$$
V_{r_1}\otimes V_{s_1}=(s_1-p^n+r_1)V_{p^n}\,\oplus\,V_{p^n-\lambda_{p^n-r_1}}\oplus\dots\oplus V_{p^n-\lambda_1}.
$$
The largest part, and its multiplicity, can now be determined as
$p^n>p^n-\lambda_{p^n-r_1}$.
\end{proof}

We now establish the way that canceling occurs in~\eqref{eq4}
when $p^n-r_1-s_1<0$ in order to obtain a decomposition (whose multiplicities are, by definition, always nonnegative).
Suppose that $p^n-r_1-s_1<0$. Then Lemma~\ref{lem1} gives
$$
V_{r_1}\otimes V_{s_1}=(r_1+s_1-p^n)V_{p^n}\,\oplus\,\bigoplus_{j=2}^l n_jV_{\nu_j}.
$$
Thus the summand corresponding to $j=1$ in the third line of~\eqref{eq4} is
\begin{equation}\label{eq8}
(r_1+s_1-p^n)\left(\bigoplus_{i=0}^{d_1}V_{(s_0-r_0+2i)p^n+p^n}\,\oplus\,\bigoplus_{i=1}^{d_1}V_{(s_0-r_0+2i)p^n-p^n}\right).
\end{equation}
When $r_0+s_0<p$, we have from Proposition~\ref{prop2} that
$d_2=d_1$ and the second line of ~\eqref{eq4} may be written as
\[
  (p^n-r_1-s_1)\bigoplus_{i=1}^{d_2}V_{(s_0-r_0+2i-1)p^n}=
  -(r_1+s_1-p^n)\bigoplus_{i=1}^{d_1}V_{(s_0-r_0+2i)p^n-p^n}.
\]
This cancels with the second sum in~\eqref{eq8}. On the other hand,
when $r_0+s_0\geq p$, we have $d_2=d_1+1$ and the second line
of~\eqref{eq4} may be written as
\begin{align*}
  (p^n-r_1-s_1)\bigoplus_{i=1}^{d_2}V_{(s_0-r_0+2i-1)p^n}&=
  -(r_1+s_1-p^n)\bigoplus_{i=1}^{d_1+1}V_{(s_0-r_0+2i-1)p^n}\\
  &=-(r_1+s_1-p^n)\bigoplus_{j=0}^{d_1}V_{(s_0-r_0+2j)p^n+p^n}.
\end{align*}
This cancels with the first sum in~\eqref{eq8}.
%\begin{eqnarray*}
%(p^n-r_1-s_1)\bigoplus_{i=1}^{d_2}V_{(s_0-r_0+2i-1)p^n}=
%\begin{cases}
%-(r_1+s_1-p^n)\bigoplus_{i=1}^{d_1}V_{(s_0-r_0+2i)p^n-p^n}\quad&\text{if $r_0+s_0<p$},\\
%-(r_1+s_1-p^n)\bigoplus_{i=0}^{d_1}V_{(s_0-r_0+2i)p^n+p^n}\quad&\text{if $r_0+s_0\geq p$}.
%\end{cases}
%\end{eqnarray*}
Therefore, after canceling in this way, (\ref{eq4}) becomes a decomposition for $V_r\otimes V_s$.

In order to complete Renaud's inductive reduction in Proposition~\ref{prop2}, we must specify what happens when $n=1$. This amounts to knowing how $V_r\otimes V_s$ decomposes when $1\leq r\leq s<p$. Such a decomposition is given in Proposition~\ref{prop3}. It can be deduced easily from \cite[Corollary~1, p.\,687]{Srinivasan1964} and Proposition~\ref{prop1}.

\begin{proposition}\label{prop3}\textup{\cite[\textsc{Theorem~1}]{Renaud1979}}
If $1\leq r\leq s\leq p$, then $V_r\otimes V_s$ decomposes as
\begin{equation}\label{eq10}
  V_r\otimes V_s=\bigoplus_{i=1}^eV_{s-r+2i-1}\oplus(r-e)V_p,
  \quad\textup{where}\quad
e=
\begin{cases}
r\quad&\text{if $r+s\leq p$},\\
p-s\quad&\text{if $r+s> p$}.
\end{cases}
\end{equation}
\end{proposition}

We will need a version of Proposition~\ref{prop3} which works independent of
the relative sizes of $r$ and $s$. This is easy when $r+s\leq p$.
In the case $r+s\geq p$ we have $e=p-s$ in~\eqref{eq10}. The subscript
$s-r+2i-1$ in equation~\eqref{eq10} equals $2p-r-s-2j+1$ where
$j:=e-i+1$ satisfies
$1\leq j\leq e$. This establishes the following symmetrised version
of~\eqref{eq10}.
\begin{corollary}\label{symcor}
If $1\leq r\leq p$ and $1\leq s\leq p$, then $V_r\otimes V_s$ decomposes as
\begin{equation}\label{eq2}
V_r\otimes V_s=
\begin{cases}
\bigoplus_{j=1}^{\min(r,s)}V_{r+s-2j+1}\quad&\text{if $r+s\leq p$},\\
(r+s-p)V_p\oplus\bigoplus_{j=1}^{p-\max(r,s)}V_{2p-r-s-2j+1}\quad&\text{if $r+s> p$}.
\end{cases}
\end{equation}
\end{corollary}

Corollary~\ref{symcor} arises in the context of tilting modules of the
special linear group $\textup{SL}(2,\F_p)$ as we now explain.
Brauer and Nesbitt~\cite{BrauerNesbitt} showed that $\textup{SL}(2,\F_p)$ has
precisely $p$ nonisomorphic indecomposable modules over the field $\F_p$, say
$V'_1,\dots,V'_p$ where $\dim(V'_r)=r$. Indeed, $V'_r$ comprises the homogeneous
polynomials in $\F_p[x,y]$ of degree~$r-1$ and
$g=\left(\begin{smallmatrix}a&b\\c&d\end{smallmatrix}\right)$ acts on $V'_r$ via
$x^g=ax+by$ and $y^g=cx+dy$. The restriction of $V'_r$ to the subgroup
$\langle\left(\begin{smallmatrix}1&1\\0&1\end{smallmatrix}\right)\rangle$
of $\textup{SL}(2,\F_p)$ gives the familiar module $V_r$. We thank
Martin Liebeck for showing us how to prove Corollary~\ref{symcor}
using tilting modules for $\textup{SL}(2,\F_p)$; see~\cite{Humphreys1975}.

\section{Proofs of Theorems \ref{thm1} and \ref{thm2}}\label{sec2}

Suppose that $r,s\geq1$ and
$V_r\otimes V_s=V_{\lambda_1}\oplus\cdots\oplus V_{\lambda_b}$ where $b=\min(r,s)$.
In this section we will prove two new results concerning the $p$-parts
$(\lambda_i)_p$ of the $\lambda_i$. We begin by proving that the $p$-parts
of $\lcm(r,s)$ and $\gcd(\lambda_1,\dots,\lambda_b)$ are equal,
i.e. $\lcm(r,s)_p=\gcd(\lambda_1,\dots,\lambda_b)_p$.
It is sometimes more convenient to prove
$\max(r_p,s_p)=\min((\mu_1)_p,\dots,(\mu_t)_p)$ by~\eqref{X}.
%This can be rephrased as
%$\max(r_p,s_p)=\min((\mu_1)_p,\dots,(\mu_t)_p)$, see~\eqref{X}.

\begin{proof}[Proof of Theorem~\textup{\ref{thm1}}]
Write $V_r\otimes V_s=V_{\lambda_1}\oplus\cdots\oplus V_{\lambda_b}$, and let $p^n$ be a $p$-power satisfying $\max(r,s)<p^n$. Write
$V_r\otimes V_s=\bigoplus_{i=1}^t m_i V_{\mu_i}$ where
$p^n\geq\mu_1>\cdots>\mu_t>0$ and each $m_i>0$ as in Notation~\ref{notn}.
We use induction on~$n$ to prove the statement~\eqref{M} paraphrasing
Theorem~\ref{thm1}. 

First suppose that $n=1$, and hence $\lcm(r,s)_p=1$. Then $r_p=s_p=1$
and hence $p\nmid\dim(V_r\otimes V_s)$. So if
$V_r\otimes V_s=\bigoplus_{i=1}^t m_iV_{\mu_i}$, then
$p\nmid\gcd(\mu_1,\dots,\mu_t)_p$. This
establishes Theorem~\ref{thm1} when $n=1$.

Suppose by induction that Theorem~\ref{thm1} holds for $\max(r,s)<p^n$ and fixed $n\geq1$. We now show that it also holds for $\max(r,s)<p^{n+1}$. Without loss of generality, assume $r\leq s<p^{n+1}$. Write $r=r_0p^n+r_1$ and $s=s_0p^n+s_1$ where $r_0,s_0,r_1,s_1\geq0$ and $r_1,s_1<p^n$. Clearly $r_0\leq s_0<p$. 
The remainder of the proof is divided into four cases.

\textbf{Case 1.} $r_1=s_1=0$. Since $r\leq s<p^{n+1}$, we deduce from $r=r_0p^n$ and $s=s_0p^n$ that $1\leq r_0\leq s_0<p$.  
Suppose that
$V_{r_0}\otimes V_{s_0}=\sum_{j=1}^l n_jV_{\nu_j}$ where $\nu_1>\cdots>\nu_\ell>0$ and
each $n_j>0$. 
It follows by \cite[\textsc{Lemma} 2.2]{Renaud1979} that $V_r\otimes V_s=\sum_{j=1}^l p^n n_jV_{p^n\nu_j}$. Hence $\mu_j=p^n\nu_j$ and $m_j=p^n n_j$ for each $j$. Now $1=\max((r_0)_p,(s_0)_p)=\min((\nu_1)_p,\dots,(\nu_l)_p)$ by induction. Multiplying this equation by $p^n$ gives
\[
  p^n=\max(r_p,s_p)=\min(p^n(\nu_1)_p,\dots,p^n(\nu_l)_p)
  =\min((\mu_1)_p,\dots,(\mu_l)_p).
\]
This is equivalent to $\lcm(r,s)_p=\gcd(\lambda_1,\dots,\lambda_b)_p$,
as desired.

\textbf{Case 2.} $r_1=0$ and $s_1>0$. In this case, $\lcm(r,s)_p=r_p=p^n$ since $p^n$ divides $r$ and $1\leq r\leq s<p^{n+1}$. Since $V_0\otimes V_{s_1}=V_0$, the partition $\lambda(r_1,s_1,p)$ has no parts, and the sum on the last line of~\eqref{eq4} is empty. Thus Proposition~\ref{prop2} gives
\begin{equation}\label{eq5}
V_r\otimes V_s=cV_{p^{n+1}}\,\oplus\,s_1\bigoplus_{i=1}^{d_1}V_{(s_0-r_0+2i)p^n}\,\oplus\,(p^n-s_1)\bigoplus_{i=1}^{d_2}V_{(s_0-r_0+2i-1)p^n},
\end{equation}
where $(c,d_1,d_2)$ is defined in Proposition~\ref{prop2}. It is clear from (\ref{eq5}) that $p^n$ divides each of $\lambda_1,\dots,\lambda_b$, and thus divides $\gcd(\lambda_1,\dots,\lambda_b)$. The following paragraph shows that $\gcd(\lambda_1,\dots,\lambda_b)_p$ divides $p^n$.

If $s_0=p-1$, then $(c,d_1,d_2)=(r+s-p^{n+1},0,1)$ in Proposition~\ref{prop2} since $r_0+s_0\geq1+s_0=p$. Any sum of the form $\bigoplus_{i=1}^0 W_i$ equals $0$,
so equation~\eqref{eq5} becomes
$$
V_r\otimes V_s=(r+s-p^{n+1})V_{p^{n+1}}\,\oplus\,(p^n-s_1)V_{(p-r_0)p^n}.
$$
Note that $r+s-p^{n+1}\geq p^n+s-p^{n+1}=s-s_0p^n=s_1>0$ and $p^n-s_1>0$. Hence $\gcd(\lambda_1,\dots,\lambda_b)$ divides $p^{n+1}-(p-r_0)p^n=r_0p^n$
and so $\gcd(\lambda_1,\dots,\lambda_b)_p$ divides $p^n$.
If $s_0\leq p-2$, then Proposition~\ref{prop2} shows that $d_2\geq d_1\geq\min(r_0,p-s_0-1)\geq1$, and hence $\gcd(\mu_1,\dots,\mu_t)$ divides $(s_0-r_0+2)p^n-(s_0-r_0+1)p^n=p^n$ in light of (\ref{eq5}). In summary, $\gcd(\mu_1,\dots,\mu_t)_p$ divides $p^n$ in both cases. Thus $\lcm(r,s)_p=\gcd(\lambda_1,\dots,\lambda_b)_p=p^n$.

\textbf{Case 3.} $r_1>0$ and $s_1=0$. In this case, $\lcm(r,s)_p=s_p=p^n$ since $p^n$ divides $s$ and $1\leq r\leq s<p^{n+1}$. 
As above, the decomposition of $V_{r_1}\otimes V_0=V_0$ is empty, so the last line of~\eqref{eq4} vanishes. Hence $V_r\otimes V_s$ equals
\begin{equation}\label{eq6}
cV_{p^{n+1}}\,\oplus\, r_1\bigoplus_{i=1}^{d_1}V_{(s_0-r_0+2i)p^n}\,\oplus\,r_1V_{(s_0-r_0)p^n}\,\oplus\,
(p^n-r_1)\bigoplus_{i=1}^{d_2}V_{(s_0-r_0+2i-1)p^n},
\end{equation}
where $(c,d_1,d_2)$ is defined in Proposition~\ref{prop2}. It is clear from (\ref{eq6}) that $p^n$ divides each $\lambda_i$ and thus divides $\gcd(\lambda_1,\dots,\lambda_b)$. The next paragraph shows that $\gcd(\lambda_1,\dots,\lambda_b)_p$
divides~$p^n$.

If $r_0=0$, then $(c,d_1,d_2)=(0,0,0)$ and (\ref{eq6}) gives $V_r\otimes V_s=r_1V_{s_0p^n}=r_1V_s$. Hence $\gcd(\mu_1,\dots,\mu_t)=\gcd(s)=s$.
Thus $\gcd(\mu_1,\dots,\mu_t)_p=s_p=p^n=\lcm(r,s)_p$, as desired. Thus we can assume $0<r_0<p$. Then $d_2\geq\min(r_0,p-s_0)\geq1$, and so $\gcd(\mu_1,\dots,\mu_t)$ divides $(s_0-r_0+1)p^n-(s_0-r_0)p^n=p^n$ in light of (\ref{eq6}). Consequently $\gcd(\mu_1,\dots,\mu_t)=p^n$, and we conclude that
for all values of $r_0$ that
$\lcm(r,s)_p=p^n=\gcd(\mu_1,\dots,\mu_t)_p$ holds.

\textbf{Case 4.} $r_1>0$ and $s_1>0$. Here $r_p=(r_1)_p$ and $s_p=(s_1)_p$,
and it follows that $\max(r_p,s_p)=\max((r_1)_p,(s_1)_p)<p^n$. 
Suppose that
$V_{r_1}\otimes V_{s_1}=\sum_{j=1}^l n_jV_{\nu_j}$ where
$p^n\geq\nu_1>\cdots>\nu_\ell>0$ and
each $n_j>0$. Our inductive hypothesis implies that
$$
  \min((\nu_1)_p,\dots,(\nu_l)_p)=\max((r_1)_p,(s_1)_p)=\max(r_p,s_p)<p^n.
$$
Assume $(\nu_k)_p=\min((\nu_1)_p,\dots,(\nu_l)_p)$, so $(\nu_k)_p$ divides
each $(\nu_j)_p$. Since $(\nu_k)_p<p^n$, Proposition~\ref{prop2} implies
that $(\nu_k)_p$ divides each $(\mu_i)_p$. Moreover by Proposition~\ref{prop2},
one of the $\mu_i$ is equal to $(s_0-r_0)p^n+\nu_k$ which has $p$-part
$(\nu_k)_p$. Hence $\min((\mu_1)_p,\dots,(\mu_t)_p)$ equals $(\nu_k)_p$ and
it follows that
$$
  \max(r_p,s_p)=\min((\nu_1)_p,\dots,(\nu_l)_p)=(\nu_k)_p=\min((\mu_1)_p,\dots,(\mu_t)_p).
$$
This is equivalent to $\lcm(r,s)_p=\gcd(\lambda_1,\dots,\lambda_b)_p$,
as desired.
\end{proof}

We now prove Theorem~\ref{thm2} which states that each part of $\lambda(r,s,p)$ with multiplicity greater than~1 must be divisible by~$p$. In other words, if $V_r\otimes V_s=\sum_{i=1}^tm_iV_{\mu_i}$ where $\mu_1>\cdots>\mu_t>0$ and
$m_i>0$ for each $i$, then $m_j>1$ implies $p$ divides $\mu_j$.

\begin{proof}[Proof of Theorem~\textup{\ref{thm2}}]
Our proof uses induction on $n$ where $\max(r,s)<p^n$.

The decomposition when $n=1$ is described by~\eqref{eq2}. The only time that
$\lambda(r,s,p)$ has a part with multiplicity more than~1 is when $r+s-p>1$. In this case the part size is~$p$. Thus Theorem~\ref{thm2} is true when $n=1$.

Next suppose that Theorem~\ref{thm2} is true for $\max(r,s)<p^n$ and
fixed $n\geq1$. We will show that it also true when
$p^n\leq\max(r,s)<p^{n+1}$. Without loss of generality, assume $r\leq s$. Set
$r=r_0p^n+r_1$ and $s=s_0p^n+s_1$ where $r_0,s_0,r_1,s_1\geq0$ and
$r_1,s_1<p^n$. Suppose that $V_{r_1}\otimes V_{s_1}=\sum_{j=1}^l
n_jV_{\nu_j}$, where $p^n\geq\nu_1>\cdots>\nu_\ell>0$ and $n_i>0$ for
each~$i$. By the inductive hypothesis, each $\nu_j$ with $n_j>1$ is a
multiple of $p$. The part sizes, or the dimensions of the
indecomposable modules, occurring in the first two lines
of~\eqref{eq4} are each divisible by~$p$. We show in the next
paragraph that the parts occurring in the last line of~\eqref{eq4} are
either distinct, or are divisible by~$p$. Once this has been
established, the inductive hypothesis completes the proof of
Theorem~\ref{thm2}.

The parts in the first sum $\bigoplus_{j=1}^\ell n_j\bigoplus_{i=0}^{d_1}V_{(s_0-r_0+2i)p^n+\nu_j}$ are distinct for distinct $(i,j)$. This is so because
\[
  (s_0-r_0+2i)p^n+\nu_j=(s_0-r_0+2i')p^n+\nu_{j'}\qquad\textup{where $0<\nu_j,\nu_{j'}\leq p^n$}
\]
implies $\nu_j=\nu_{j'}$ and hence $j=j'$; and then $i=i'$ follows. A similar
argument shows that the parts in the second sum
$\bigoplus_{j=1}^\ell n_j\bigoplus_{i=1}^{d_1}V_{(s_0-r_0+2i)p^n-\nu_j}$
are distinct. If a part from the first sum equals a part from
the second sum, then there exist integers $i,i',j,j'$ satisfying
\[
  (s_0-r_0+2i)p^n+\nu_j=(s_0-r_0+2i')p^n-\nu_{j'}.
\]
Hence $2(i'-i)p^n=\nu_j+\nu_{j'}$. However, $0<\nu_j,\nu_{j'}\leq p^n$ implies
$0<\nu_j+\nu_{j'}\leq 2p^n$ and hence $\nu_j+\nu_{j'}$ is divisible by $2p^n$
which is possible only when $\nu_j=\nu_{j'}=p^n$. Thus $2(i'-i)p^n=2p^n$ and
$i'-i=1$. Consequently, a part
from the first sum equals a part from the second sum only when the part sizes
are divisible by~$p^n$ (and hence by~$p$). As remarked above, induction
now completes the proof.
\end{proof}

\section{Iima and Iwamatsu's Algorithm}\label{sec3}

Assume $1\leq r\leq s$ throughout this section. For $k=1,\dots,r$,
define $D_k=D_k(r,s)$ to be the determinant of the $k\times k$ matrix
$A_k$ whose $(i,j)$th entry is $\binom{r+s-2k}{s+i-j-k}$ for $0\leq
i,j<k$.  Given nonnegative integers $M$ and $N$, the matrix
$\begin{pmatrix}\binom{M}{N+i-j}\end{pmatrix}_{0\leq i,j,<k}$ has
determinant $\prod_{i=0}^{k-1}\binom{M+i}{N}/\binom{N+i}{N}$, see
~\cite[p.\;355]{Robert1994}. Setting $M:=r+s-2k$ and $N:=s-k$ gives
the following closed formula
%The following closed formula is proved in~\cite[p.\;355]{Robert1994}
\begin{equation}\label{eq1}
D_k(r,s)=\prod_{i=0}^{k-1}\frac{\binom{r+s-2k+i}{s-k}}{\binom{s-k+i}{s-k}},\quad\textup{where }1\leq k\leq r.
\end{equation}
Even though the right-hand side of~\eqref{eq1} looks like a rational number,
$D_k(r,s)$ is an integer (as it is the determinant of a matrix
with integer entries).
Set $D_0(r,s):=1$, and note that $D_r(r,s)=1$. For $k=0,1,\dots,r$, define
\begin{equation}\label{del}
\delta_k=\delta_k(r,s,p)=
\begin{cases}
0\quad&\text{if $D_k(r,s)\equiv0\pmod{p}$,}\\
1\quad&\text{if $D_k(r,s)\not\equiv0\pmod{p}$.}
\end{cases}
\end{equation}
Thus $\delta_k=1$ says that $A_k$ is invertible when viewed as a
matrix over $\F_p$. In other words, $\delta_k=1$ says that $A_k$ has
full $p$-rank. Iima and Iwamatsu~\cite{II2009} found a way to
construct $\lambda(r,s,p)$ from the $\{0,1\}$-sequence
$\delta_0(r,s,p),\delta_1(r,s,p),\dots,\delta_r(r,s,p)$. This
constrains the number of choices of $\lambda(r,s,p)$ as described
in~\cite{GPX}. Note that $\delta_0=\delta_r=1$ by our convention that
$D_0(r,s)=1$ and $D_r(r,s)=1$. 

For $1\leq k\leq r$, if $\delta_k=1$, let $\ell(k)$ be the smallest
positive integer such that $\delta_{k-\ell(k)}=1$. Note that $\ell(k)$
is well defined since $\delta_0=1$, and $\ell(k)\leq k$. The following
Proposition is proved by the results in \cite{II2009} preceding and
including Theorem 2.2.9.

\begin{proposition}\label{prop4}\textup{\cite[Theorem 2.2.9]{II2009}}
Suppose $1\leq r\leq s$, and use the above notation for $\delta_k$ and
$\ell(k)$ for $1\leq k\leq r$. Then the parts of the Jordan partition
$\lambda(r,s,p)$ can be computed via the following recurrence
where~$k$ decreases from~$r$ to~\textup{1}
\begin{equation*}
\lambda_k=
\begin{cases}
  r+s-2k+\ell(k)\quad&\text{if $\delta_k=1$},\\
  \lambda_{k+1}\quad&\text{if $\delta_k=0$}.
\end{cases}
\end{equation*}
\end{proposition}

The next proposition is a reformulation of Proposition~\ref{prop4} in
the language of Green ring results. While this result essentially
appears in~\cite{II2009}, its proof is long and somewhat complicated,
so we prefer to give our own proof. Recall the
definition~\eqref{del}~of~$\delta_k$.

\begin{proposition}\label{prop5}\textup{\cite[Theorem 2.2.9]{II2009}} Suppose $1\leq r\leq s$, and all the values of $k$ satisfying $\delta_k(r,s,p)=1$ are $0=k_0<k_1<\dots<k_t=r$.
Then $V_r\otimes V_s$ decomposes as
\begin{equation}\label{eq12}
V_r\otimes V_s=\bigoplus_{i=1}^t(k_i-k_{i-1})V_{r+s-k_i-k_{i-1}}.
\end{equation}
\end{proposition}

\begin{proof}
Because $k_{i-1}<k_i$ for $1\leq i\leq t$, we have
\[
  \delta_{k_{i-1}}=1,\quad\delta_{k_{i-1}+1}=\delta_{k_{i-1}+2}=\cdots=\delta_{k_i-1}=0,\quad\textup{and}\quad\delta_{k_i}=1.
\]
Then $\ell(k_i)=k_i-k_{i-1}$. Appealing to second case of the recurrence in Proposition~\ref{prop4} gives
\[
  \lambda_{k_{i-1}+1}=\cdots=\lambda_{k_i-1}=\lambda_{k_i},
\]
and appealing to the first case of Iima and Iwamatsu's recurrence gives
\[
  \lambda_{k_i}=r+s-2k_i+\ell(k_i)=r+s-2k_i+(k_i-k_{i-1})=r+s-k_i-k_{i-1}.
\]
This proves that $V_r\otimes V_s=\bigoplus_{i=1}^t(k_i-k_{i-1})V_{r+s-k_i-k_{i-1}}$.
Since $0=k_0<k_1<\dots<k_t=r$, different values of~$i$
give different values of $r+s-k_i-k_{i-1}$. Hence~\eqref{eq12} is indeed
a decomposition, with distinct parts and positive multiplicities, as claimed.
\end{proof}

It follows from Proposition~\ref{prop5} that the multiplicities
$m_i=k_i-k_{i-1}$, $1\leq i\leq t$,
determine the distinct part sizes $\mu_i=r+s-k_i-k_{i-1}$, $1\leq i\leq t$, and
conversely. Theorem~\ref{thm5} shows how $\mu_1,\dots,\mu_t$ determine
$m_1,\dots,m_t$ via explicit formulas.

\begin{proof}[Proof of \textup{Theorem~\ref{thm5}}]
Our strategy is to prove McFall's
conjecture~\cite[Conjecture\;2]{McFall1980} that
\begin{equation}\label{R}
  m_1=r+s-\mu_1\qquad\textup{and}\qquad m_i=\mu_{i-1}-\mu_i-m_{i-1}\quad 
  \textup{\ for\ }1< i\leq t.
\end{equation}
A straightforward calculation shows that the formula~\eqref{C2}, satisfies
this recurrence relation, and hence that~\eqref{R} implies the
equalities in~\eqref{C2}. Rearranging~\eqref{R} gives
a recurrence relation for computing the $\mu_i$ from the $m_j$, namely
\begin{equation}\label{E:mTOmu}
  \mu_1=r+s-m_1\qquad\textup{and}\qquad \mu_i=\mu_{i-1}-m_i-m_{i-1}\quad 
  \textup{\ for\ }1< i\leq t.
\end{equation}
A further simple calculation shows that the formulas~\eqref{E:McF} are equivalent
to the rearranged recurrence relation~\eqref{E:mTOmu}.

As noted above, Proposition~\ref{prop5} shows that $m_i=k_i-k_{i-1}$ and $\mu_i=r+s-k_i-k_{i-1}$
for $1\leq i\leq t$. The initial condition of~\eqref{R} follows from $k_0=0$ as
\[
  r+s-\mu_1=r+s-(r+s-k_1-k_0)=k_1+k_0=k_1-k_0=m_1.
\]
For $1<i\leq t$, the inductive step of~\eqref{R} also follows easily
as $\mu_{i-1}-\mu_i-m_{i-1}$ equals
\[
  (r+s-k_{i-1}-k_{i-2})-(r+s-k_i-k_{i-1})-(k_{i-1}-k_{i-2})=k_i-k_{i-1}=m_i.
\]
This establishes the recurrence relation~\eqref{R}, and thereby proves Theorem~\ref{thm5}.
\end{proof}

The $p$-divisibility of the integers $D_0(r,s),D_1(r,s),\dots,D_r(r,s)$
plays a central role in Iima and Iwamatsu's algorithm. 
Kummer's theorem~\cite{Granville1997} states that the power of a prime~$p$
dividing $\binom{m}{n}$ is the number of `carries' required to add $m$ and
$n-m$ in base-$p$. This can be used to compute the
largest $p$-power dividing the numerator and denominator of~\eqref{eq1}.
The following lemma gives a more direct approach, and it has a nice application
in Section~\ref{sec4}.

\begin{lemma}\label{lem2}
Suppose $1\leq r\leq s$, and let $D_k(r,s)$ be as in~\eqref{eq1} with
$D_0(r,s)=D_r(r,s)=1$.
\begin{itemize}
\item[(a)] If $0\leq k\leq r$, then $\binom{s}{s-k}D_{k+1}(r+1,s+1)=\binom{r+s-k}{s-k}D_k(r,s)$.
\item[(b)] If $0\leq k\leq r-1$, then $\binom{s}{s-k}D_{k+1}(r,s+1)=\binom{r+s-2k-1}{s-k}D_k(r,s)$.
\item[(c)] If $0\leq k\leq r-1$, then $\binom{r+s-k-1}{k}D_{k+1}(r,s)=\binom{r+s-2k-2}{s-k-1}D_k(r,s)$.
\end{itemize}
\end{lemma}

\begin{proof}
The proof is by direct calculation using~\eqref{eq1}. Part~(a) follows from
$$
D_{k+1}(r+1,s+1)=\prod\limits_{i=0}^{k}\frac{\binom{r+s-2k+i}{s-k}}{\binom{s-k+i}{s-k}}
=\frac{\binom{r+s-k}{s-k}}{\binom{s}{s-k}}\prod\limits_{i=0}^{k-1}\frac{\binom{r+s-2k+i}{s-k}}{\binom{s-k+i}{s-k}}=\frac{\binom{r+s-k}{s-k}}{\binom{s}{s-k}}D_k(r,s).
$$
The proof of part~(b) follows from the formula~\eqref{eq1} and the identities
$\binom{m-1}{n}=\frac{m-n}{m}\binom{m}{n}$ and
$\binom{m}{n}\prod_{i=0}^{k-1}\frac{m-n-k+i+1}{m-k+i+1}=\binom{m-k}{n}$
\begin{eqnarray*}
D_{k+1}(r,s+1)&=&\frac{\binom{r+s-k-1}{s-k}}{\binom{s}{s-k}}\prod\limits_{i=0}^{k-1}\frac{\binom{r+s-2k-1+i}{s-k}}{\binom{s-k+i}{s-k}}\\
&=&\frac{\binom{r+s-k-1}{s-k}}{\binom{s}{s-k}}\prod\limits_{i=0}^{k-1}\frac{\binom{r+s-2k+i}{s-k}(r-k+i)}{\binom{s-k+i}{s-k}(r+s-2k+i)}
=\frac{\binom{r+s-2k-1}{s-k}}{\binom{s}{s-k}}D_k(r,s).
\end{eqnarray*}
To prove part~(c), we use the identity
$$
D_k(r,s)=\prod\limits_{i=0}^{k-1}\frac{(r+s-2k+i)!i!}{(s-k+i)!(r-k+i)!}.
$$
We now write $D_{k+1}(r,s)$ in terms of the above product
\begin{eqnarray*}
D_{k+1}(r,s)&=&\binom{r+s-2k-2}{s-k-1}\prod\limits_{i=1}^{k}\frac{\binom{r+s-2k-2+i}{s-k-1}}{\binom{s-k-1+i}{s-k-1}}\\
&=&\binom{r+s-2k-2}{s-k-1}\prod\limits_{i=1}^{k}\frac{(r+s-2k-2+i)!i!}{(s-k-1+i)!(r-k-1+i)!}\\
&=&\binom{r+s-2k-2}{s-k-1}\prod\limits_{i=0}^{k-1}\frac{(r+s-2k-1+i)!(i+1)!}{(s-k+i)!(r-k+i)!}\\
&=&\binom{r+s-2k-2}{s-k-1}\prod\limits_{i=0}^{k-1}\frac{(r+s-2k+i)!i!(i+1)}{(s-k+i)!(r-k+i)!(r+s-2k+i)}\\
&=&\frac{\binom{r+s-2k-2}{s-k-1}}{\binom{r+s-k-1}{k}}\prod\limits_{i=0}^{k-1}\frac{(r+s-2k+i)!i!}{(s-k+i)!(r-k+i)!}.
\end{eqnarray*}
\end{proof}

\section{Results for \texorpdfstring{$|r-s|$}{} at most one}\label{sec4}

In this section, we prove several results when $|r-s|\leq1$. First, we
determine the smallest part of $\lambda(r,r,p)$, and its multiplicity.
As usual, we denote the $p$-part of a nonzero integer~$r$ by $r_p$.

Lucas' theorem (see~\cite{Granville1997}) is a useful number-theoretic result
for proving $D_{r-p^k}(r,r)\not\equiv 0\pmod p$, or $\delta_{r-p^k}(r,r,p)=1$
as in~\eqref{del}. This
theorem says that $\binom{m}{n}\equiv\prod_{i\geq0}\binom{m_i}{n_i}\pmod p$ where
$m=\sum_{i\geq0} m_ip^i$ and $n=\sum_{i\geq0} n_ip^i$ are the base-$p$
expansions of $m$ and $n$, respectively. The base-$p$ `digits' $m_i,n_i$ satisfy
$0\leq m_i,n_i<p$. Note that $\binom{m_i}{n_i}=0$
if $m_i<n_i$, and $\binom{m_i}{0}=1$. Thus the infinite product 
$\prod_{i\geq0}\binom{m_i}{n_i}$ is finite, as
$\binom{m_i}{n_i}=1$ for sufficiently large $i$.

\begin{theorem}
The smallest part of $\lambda(r,r,p)$
is $r_p$, and it occurs with multiplicity~$r_p$. Using
Notation~\textup{\ref{notn}} and $b=\min(r,r)=r$, this says that
$\lambda_r=r_p=\mu_t$ and $m_t=r_p$.
\end{theorem}

\begin{proof}
Suppose that $r_p=p^k$. Then $r=ap^k$ with $a_p=1$. By virtue of
Proposition~\ref{prop5} (or by~\ref{prop4}), it suffices to show that $D_{r-j}(r,r)\equiv0\pmod{p}$ for $0<j<p^k$ and $D_{r-p^k}(r,r)\not\equiv0\pmod{p}$, since $D_r(r,r)=1$.

Using formula~\eqref{eq1} and canceling gives
\begin{equation}\label{eq13}
D_{r-p^k}(r,r)=\prod\limits_{i=0}^{(a-1)p^k-1}\frac{\binom{2p^k+i}{p^k}}{\binom{p^k+i}{p^k}}=\prod\limits_{i=0}^{p^k-1}\frac{\binom{ap^k+i}{p^k}}{\binom{p^k+i}{p^k}}.
\end{equation}
For $0\leq i<p^k$, Lucas' theorem shows
$\binom{ap^k+i}{p^k}\equiv\binom{a}{1}\binom{i}{0}\equiv a\pmod p$. The
numerator in~\eqref{eq13} is
$\prod\limits_{i=0}^{p^k-1}\binom{ap^k+i}{p^k}\equiv a^{p^k}\not\equiv0\pmod p$.
Thus
$D_{r-p^k}(r,r)\not\equiv0\pmod p$, as desired.
[Incidentally, $D_{r-p^k}(r,r)\equiv a\pmod p$ as $a^{p-1}\equiv1\pmod p$
and $\binom{p^k+i}{p^k}\equiv1\pmod p$ holds for $0\leq i\leq p^k-1$.]

One way to prove that $D_{r-j}(r,r)\equiv0\pmod p$ is to show that~$p$ divides
the numerator of~\eqref{eq1} to a higher power than the denominator.
This requires stronger results than Lucas' theorem. (Kummer
proved that the power of $p$ dividing $\binom{m}{n}$ is the number of
$i$ for which $m_i<n_i$, see~\cite{Granville1997}.)
A simpler approach involves using Lemma~\ref{lem2}(a).
Suppose $0<j<p^k$. Then Lemma~\ref{lem2}(a) gives
\begin{equation}\label{eq3}
  \binom{ap^k}{j}D_{r-j+1}(r+1,r+1)=\binom{ap^k+j}{j}D_{r-j}(r,r).
\end{equation}
Again by Lucas' theorem, $\binom{ap^k}{j}\equiv0\pmod{p}$ and
$\binom{ap^k+j}{j}\equiv1\pmod{p}$, so (\ref{eq3}) implies that
$D_{r-j}(r,r)\equiv0\pmod{p}$. The proof is thus completed.
\end{proof}

For the rest of this section, we establish a decomposition formula
for $V_r\otimes V_s$ when $|r-s|\leq1$ and $p=2$. 
The following proposition shortens the proof of Theorem~\ref{thm3}.
This result already appears in~\cite[Corollary~1]{Barry1980}, albeit
in a slightly less general form.

\begin{proposition}\label{P2}
Suppose $1\leq r\leq p^n$ and $1\leq s\leq p^n$. Then
\[
  V_{p^n-r}\otimes V_{p^n-s}
    =\max(p^n-r-s,0)V_{p^n}\,\oplus\,\left(V_r\otimes V_s\right).
\]
\end{proposition}

\begin{proof}
Let $V_r\otimes V_s= V_{\lambda_1}\oplus\cdots\oplus V_{\lambda_b}$
where $b:=\min(r,s)$. Proposition~\ref{prop1} yields
\[ %\begin{equation}\label{E:P1}
  V_{p^n-r}\otimes V_s
  =(s-b)V_{p^n}\oplus V_{p^n-\lambda_b}\oplus\cdots\oplus V_{p^n-\lambda_1}.
\] %\end{equation}
Since $s\leq p^n$ and $p^n-r\leq p^n$, applying Proposition~\ref{prop1}
to~$V_s\otimes V_{p^n-r}$ gives
\[
  V_{p^n-s}\otimes V_{p^n-r}
    =\left[(p^n-r)-\min(p^n-r,s)\right]V_{p^n}\,\oplus\,V_{\lambda_1}\,\oplus\,\cdots\,\oplus\,V_{\lambda_b}\oplus(s-b)V_0.
\]
Replacing the expression in square brackets with $\max(p^n-r-s,0)$,
and omitting the last summand gives the desired decomposition
of~$V_{p^n-r}\otimes V_{p^n-s}$.
\end{proof}

Our decompositions for $V_r\otimes V_r$ and $V_r\otimes V_{r+1}$ when $p=2$
depend on
a {\it `consecutive-ones-binary-expansion'} which~we now define.
The binary number $(1\cdots 10\cdots 0)_2$ with~$m$ consecutive ones, and~$n$
consecutive zeros, equals $2^{m+n}-2^n$. Thus a binary expansion
$r=\sum_{i=1}^\ell\sum_{j=b_i}^{a_i-1}2^j$ with $\ell$ groups of consecutive ones and
$a_1>b_1>\dots>a_\ell>b_\ell\geq0$ simplifies to $r=\sum_{i=1}^\ell(2^{a_i}-2^{b_i})$. 
We call an alternating sum $r=\sum_{i=1}^k(-1)^{i-1}2^{e_i}$,
with decreasing powers of~$2$ and minimal length, the
{\it `consecutive-ones-binary-expansion'} of~$r$. Minimal length implies
$e_{k-1}>e_k+1$ when $k>1$: otherwise $2^{e_k+1}-2^{e_k}$ can be replaced
by~$2^{e_k}$. Note that $r=\sum_{i=1}^\ell(2^{a_i}-2^{b_i})$ is the
consecutive-ones-binary-expansion if and only if $a_\ell>b_\ell+1$.
For example, $4=2^2, 5=2^3-2^2+2^0, 6=2^3-2^1$ are
consecutive-ones-binary-expansions. The partial sums
$r_j=\sum_{i=j+1}^k(-1)^{i-j-1}2^{e_i}$, $0\leq j\leq k$, associated to the
consecutive-ones-binary-expansion $r=\sum_{i=1}^k(-1)^{i-1}2^{e_i}$ satisfy
$r_0=r$, $r_k=0$,
$r_i=2^{e_{i+1}}-r_{i+1}$, and $1\leq r_i\leq 2^{e_i}$ for $0\leq i<k$.
Also $e_i=\lceil\log_2(r_i)\rceil$ for $1\leq i< k$.

The following theorem originally appeared as Theorems 14 and 16 of~\cite{GPX2}.
We are grateful to M.\,J.\,J. Barry who showed us a simplified
proof of Theorem~\ref{thm3}, and we thank him for his permission to include
(a modified version of) his proof.

\begin{theorem}\label{thm3}
Suppose $\Char(F)=2$ and $r=\sum_{i=1}^k(-1)^{i-1}2^{e_i}$ is the
consecutive-ones-binary-expansion of $r$ where $e_1>\cdots>e_k\geq0$.
Set $r_j=\sum_{i=j+1}^k(-1)^{i-j-1}2^{e_i}$ for $0\leq j\leq k$ where
$r_k=0$. Then $V_r\otimes V_r$ and $V_r\otimes V_{r+1}$ decompose over~$F$~as
\begin{equation}\label{eq14}
  V_r\otimes V_r=\bigoplus_{i=1}^k(2^{e_i}-2r_i)V_{2^{e_i}}\quad\textup{and}\quad
  V_r\otimes V_{r+1}=\bigoplus_{i=1}^k(2^{e_i}-2r_i+(-1)^{i-1})V_{2^{e_i}}.
\end{equation}
In particular, each part of $\lambda(r,r,2)$ is a power of $2$.
Furthermore, parts not equal to~$1$ have even multiplicities, and~$1$ has
multiplicity at most~$1$. Also each part of $\lambda(r,r+1,2)$ is a power of $2$
greater than~$1$.
\end{theorem}

\begin{proof}
We prove~\eqref{eq14} using induction on $k$. The decomposition for
$V_r\otimes V_r$ holds when $k=1$ by~\cite[(2.7d)]{Green1962}.
Suppose now that $k>1$ and
$V_{r_1}\otimes V_{r_1}=\bigoplus_{i=2}^k(2^{e_i}-2r_i)V_{2^{e_i}}$ holds by induction.
Observe that $r_1\leq 2^{e_2}$ so $2r_1\leq 2^{e_2+1}\leq 2^{e_1}$,
and $2^{e_1}-2r_1\geq0$. Proposition~\ref{P2} implies
\begin{align}\label{E98}
  V_r\otimes V_r&=V_{2^{e_1}-r_1}\otimes V_{2^{e_1}-r_1}
                  &&\textup{as $r=2^{e_1}-r_1$,}\\
                &=(2^{e_1}-2r_1)V_{2^{e_1}}\oplus\;(V_{r_1}\otimes V_{r_1})
                  &&\textup{as $2^{e_1}-2r_1\geq0$.}\nonumber
\end{align}
The decomposition for $V_r\otimes V_{r+1}$ holds when $k=0$,
and when $k=1$ by~\cite[(2.7d)]{Green1962}. Suppose $k>1$, and
$V_{r_2}\otimes V_{r_2+1}=\bigoplus_{i=3}^k(2^{e_i}-2r_i+(-1)^{i-1})V_{2^{e_i}}$
is valid by induction. As above, $2^{e_1}-2r_1\geq0$ obtains.
Moreover, $2^{e_2}-2r_2-1\geq0$ is true.
This is easily seen when $k=2$, it follows from
$r_2\leq2^{e_3}$ using $e_{k-1}>e_k+1$ when $k=3$, and for $k>3$ it follows
from $r_2<2^{e_3}$ using
$e_3+1\leq e_2$. Applying the equations $r=2^{e_1}-r_1$, $r_1=2^{e_2}-r_2$, and
Proposition~\ref{P2} twice, now gives
\begin{align}\label{E99}
  V_r\otimes V_{r+1}&=V_{r+1}\otimes V_r\nonumber\\
                  &=V_{2^{e_1}-(r_1-1)}\otimes V_{2^{e_1}-r_1}\nonumber\\
                &=(2^{e_1}-2r_1+1)V_{2^{e_1}}\oplus\;(V_{r_1-1}\otimes V_{r_1})\\
                &=(2^{e_1}-2r_1+1)V_{2^{e_1}}
                  \oplus\;(V_{2^{e_2}-r_2-1}\otimes V_{2^{e_2}-r_2})\nonumber\\
                &=(2^{e_1}-2r_1+1)V_{2^{e_1}}\oplus (2^{e_2}-2r_2-1)V_{2^{e_2}}
                  \oplus\;(V_{r_2}\otimes V_{r_2+1}).\nonumber
\end{align}
Thus~\eqref{eq14} follows from~\eqref{E98} and~\eqref{E99} by induction on $k$.
As a by-product we have proved that the multiplicities in~\eqref{eq14}
are nonnegative, and~\eqref{eq14} is a valid decomposition.
\end{proof}

To illustrate Theorem~\ref{thm3} take $r=5$. Then $r$ has 
consecutive-ones-binary-expansion $5=2^3-2^2+2^0$. Substituting
$r_1=3$, $r_2=1$, $r_3=0$ into~\eqref{eq14} gives
\[
  V_5\otimes V_5=2V_8\oplus 2V_4\oplus V_1\quad\textup{and}\quad
  V_5\otimes V_6=3V_8\oplus V_4\oplus 2V_1
\]
over a field of characteristic~2.
The novelty of Theorem~\ref{thm3} is the decomposition~\eqref{eq14}. The parity
of the multiplicities were already known to Gow
and Laffey~\cite[Corollaries~1~and~2]{GL}.

\noindent\textsc{Acknowledgements.}
We would like to thank M.\,J.\,J. Barry for showing us a simplified
proof of Theorem~\ref{thm3} and allowing us to include his proof.
The first and second authors acknowledge the support of the Australian
Research Council Discovery Grants DP110101153 and DP130100106, and
the third author
would like to thank the China Scholarship Council for its financial support.
We also thank Martin Liebeck for his remarks concerning tilting modules.

\end{document}